\documentclass[a4paper,reqno,12pt]{amsart}
\usepackage{amsmath,amsfonts,amssymb,pdfsync}
\usepackage[latin1]{inputenc}
\usepackage{color}
\usepackage{hyperref}         
 \usepackage{times}

\newtheorem{theorem}{Theorem}[section]
\newtheorem{definition}[theorem]{Definition}

\newtheorem{lemma}[theorem]{Lemma}
\newtheorem{corollary}[theorem]{Corollary}
\newtheorem{remark}[theorem]{Remark}

\newcommand{\C}{\mathbb{C}}
\def \R {\mathbb{R}}

\numberwithin{equation}{section}

\title[Pullback attractors for a singularly nonautonomous plate equation]{Pullback attractors for a singularly nonautonomous plate equation}

\author[V. L. Carbone]{Vera Lúcia Carbone}

\author[M. J. D. Nascimento]{Marcelo José Dias Nascimento}

\author[K. Schiabel-Silva]{Karina Schiabel-Silva}

\author[R. P. Silva]{Ricardo Parreira da Silva}

\address[V. L. Carbone]{Departamento de Matem\'atica, Universidade Federal de S\~{a}o
Carlos, 13565-905 S\~{a}o
Carlos SP, Brazil}
\email{carbone@dm.ufscar.br}

\address[M. J. D. Nascimento]{Departamento de Matem\'atica, Universidade Federal de S\~{a}o
Carlos, 13565-905 S\~{a}o
Carlos SP, Brazil}
\email{marcelo@dm.ufscar.br}

\address[K. Schiabel-Silva]{Departamento de Matem\'atica, Universidade Federal de S\~{a}o
Carlos, 13565-905 S\~{a}o
Carlos SP, Brazil}
\email{schiabel@dm.ufscar.br}

\address[R. P. Silva]{Departamento de Matem\'atica, IGCE-UNESP, Caixa Postal 178, 13506-700 Rio Claro SP, Brazil}
\email{rpsilva@rc.unesp.br}

\begin{document}

\begin{abstract}
We consider the family of singularly nonautonomous plate equation with structural dam\-ping
\[
u_{tt} + a(t,x)u_{t} + (- \Delta) u_{t} +  (-\Delta)^{2} u + \lambda u = f(u),
\]
in a bounded domain $\Omega \subset \R^n$, with Navier boundary conditions. When the nonlinearity $f$ is dissipative we show that this problem is globally well posed in $H^2_0(\Omega) \times L^2(\Omega)$ and has a family of pullback attractors which is upper-semicontinuous under small perturbations of the damping $a$.

\end{abstract}

\maketitle


\section{Introduction}\label{sec:intr}

We are concerned with the nonautonomous plate equation

\begin{equation}\label{eq:plate}
\begin{cases}
u_{tt} + a_\epsilon(t,x)u_{t} +  (- \Delta) u_{t} +  (-\Delta)^{2} u + \lambda u = f(u)  & \qquad \hbox{ in } \Omega, \\
u= \Delta u = 0 & \qquad \hbox{ on }  \partial \Omega,
\end{cases}
\end{equation}
where $\Omega$ is a bounded smooth domain in $\R^{n}$, $\lambda >0$ and $f:\R \to \R$ is a dissipative nonlinearity with growth conditions which will be specified later. The map $\R \ni t \mapsto a_\epsilon(t,\cdot) \in L^\infty(\Omega)$ is supposed to be H\"older continuous with exponent $0<\beta <1$ and constant $C$, uniformly in $\epsilon \in [0,1]$. Moreover, we suppose that there are positive constants $\alpha_0, \alpha_1 \in \R$ such that 
$\alpha_0 \leqslant a_\epsilon(t,x) \leqslant \alpha_1,$
for $(t,x) \in \R \times \Omega$, $\epsilon \in [0,1]$, and we assume the convergence $a_\epsilon(t,x) \to a_0(t,x)$ as $\epsilon \to 0$, uniformly in $\R \times \Omega$.

The subject of this paper is to analyze the asymptotic behavior of the equation \eqref{eq:plate}, in the energy space $H^2_0(\Omega) \times L^2(\Omega)$, from the pullback attractors theory point of view, \cite{TLR, ChV}, and also to derive some stability properties for the ``pullback structures'' for small values of the parameter~ $\epsilon$.

The investigation of the asymptotic behavior of nonlinear dissipative equations subjected to perturbations on parameters has been extensively studied in the last two decades, with the goal of understanding how the variation of some parameters in the models of the natural sciences can determine the evolution of their state.

In the literature the asymptotic behavior and regularity properties of solutions of  second order differential equations
\begin{equation}\label{eq:abst-sec-order}
u_{tt} + Au_t + Bu = f(t,u),
\end{equation}
where $A$ and $B$ are self-adjoint operators in a Hilbert space $X$ and satisfy some monotonicity properties, has been subject of recent and intense research. Such problems arise on models of vibration of elastic systems and was extensively  studied in \cite{Chen, Trig, DiBlasio, EM, Haraux, Haraux1, Huang, Liu, Xiao, Zhong} and in the references given there. It is important to observe that in such works the linear operators it is not time dependent. However, to study the problem \eqref{eq:plate} we will deal  with equations where the linear operators are time dependent in the form
\begin{equation}\label{eq:abst-sec-order-sing-non-aut}
u_{tt} + A(t)u_t + B(t)u = f(t,u).
\end{equation}
We emphasize this particularity using the term singularly non-autonomous. To deal with such equations we will need a concise existence theory as well continuation results of solutions that will be done in the Section 2. In the Section 3 we obtain some energy estimates necessary to guarantee that the solution operator for \eqref{eq:plate} defines an evolution process which is strongly bounded dissipative. In the Section 4 we present basic definitions and the abstract framework of  the theory of pullback attractors and we prove existence of pullback attractors for the problem \eqref{eq:plate} as well their upper-semicontinuity is $\epsilon=0$.

\section{Setting of the problem}

If $A:=(-\Delta)^{2}$ denote the biharmonic operator with domain $D(A)=\{u \in H^{4}(\Omega) \cap H^{1}_0(\Omega) : \Delta u_{|\partial \Omega}=~0\}$, it is well known that  $A$ is a positive self-adjoint operator in $L^2(\Omega)$ with compact resolvent and therefore $-A$ generates a compact analytic semigroup in $\mathcal{L}(L^2(\Omega))$. Let us to consider, for $\alpha \geqslant 0$, the scale of Hilbert spaces $E^\alpha:= \big(D(A^\alpha), \|A^\alpha \cdot \|_{L^2(\Omega)} + \| \cdot \|_{L^2(\Omega)} \big)$. It is of special interest the case $\alpha=\frac{1}{2}$, where $-A^{\frac{1}{2}}$ is the Laplace operator with homogeneous Dirichlet boundary conditions, ie, $A^{\frac{1}{2}}= -\Delta$ with domain $E^{\frac{1}{2}} = H^2(\Omega) \cap H^1_0(\Omega)$ endowed with the norm $\|u\|_{E^\frac{1}{2}}= \|\Delta u\|_{L^2(\Omega)} + \| u\|_{L^2(\Omega)}$. 

Setting the Hilbert space $X^0 := E^\frac{1}{2} \times E^0$, let $\mathcal{A}: D(\mathcal{A}) \subset X^0 \to X^0$ be the elastic operator
$$
\mathcal{A}:= \left[\begin{array}{cc}
 0 & -I  \\
 A + \lambda I &  A^{\frac{1}{2}}
  \end{array}\right],
$$
with domain $D(\mathcal{A}):= E^1 \times  E^\frac{1}{2}$. It is well known that this operator generates a compact analytic semigroup in $X^0$, see for instance \cite{CC1, Trig, Haraux}. Writing $\mathcal{A}_\epsilon(t):= \mathcal{A} + \mathcal{B}_\epsilon(t)$, where $\mathcal{B}_\epsilon(t)$ is the uniformly bounded operator given by
\[
\mathcal{B}_\epsilon(t) := \left[\begin{array}{cc}
  0 & 0 \\
  0 &  a_\epsilon(t,\cdot)I \end{array}\right],
\]
it follows that $\mathcal{A}_\epsilon(t)$ is also a sectorial operator in $X^0$, with domain $D(\mathcal{A}_\epsilon(t))=D(\mathcal{A})$ (as a vector space) independent of $t$ and $\epsilon$. We observe that from the definition of $\mathcal{A}_\epsilon(t)$, it follows easily from Open Mapping Theorem that $X^1:=(D(\mathcal{A}), \|\mathcal{A} \cdot \|_{X^0} + \| \cdot \|_{X^0})$ is isomorphic to the space $X^1(t):= (D(\mathcal{A}), \|\mathcal{A}_\epsilon(t) \cdot~\|_{X^0} + \| \cdot \|_{X^0})$, uniformly in $t \in \R$ and $\epsilon \in [0,1]$, since we have
$$
\left\| \mathcal{A}_\epsilon(t) \left[\begin{array}{c}
  u\\
  v \end{array}\right]  \right\|_{X^0} +  \left\| \left[\begin{array}{c}
  u\\
  v \end{array}\right] \right\|_{X^0}  \leqslant  \left\| \mathcal{A} \left[\begin{array}{c}
  u\\
  v \end{array}\right] \right\|_{X^0} +\;  (\alpha_1 +1) \left\| \left[\begin{array}{c}
  u \\
  v \end{array}\right] \right\|_{X^0} \simeq  \left\| \left[\begin{array}{c}
  u \\
  v \end{array}\right] \right\|_{X^1}.
$$

Next we introduce another scale of Hilbert spaces in order to rewrite the equation \eqref{eq:plate} as an ordinary differential equation in a suitable space. We consider $X^\alpha := \big( D(\mathcal{A}^\alpha), \|\mathcal{A}^\alpha \cdot \|_{X^0} + \| \cdot \|_{X^0} \big)$, so by complex interpolation we have $X^\alpha=[X^0,X^1]_{\alpha} = E^\frac{\alpha+1}{2} \times E^\frac{\alpha}{2}$, and the $\alpha$-realization $\mathcal{A}_{\epsilon_\alpha}(t)$ of $\mathcal{A}_\epsilon(t)$ in $X^\alpha$ is an isometry of $X^{\alpha +1}$  onto $X^\alpha$. Also, the sectorial operator $\mathcal{A}_{\epsilon_\alpha}(t): X^{\alpha+1} \subset X^\alpha \to X^\alpha$ in $X^\alpha$ generates a compact analytic semigroup $\{e^{-\mathcal{A}_{\epsilon_\alpha}(t) s}: s\geqslant 0\}$ in $\mathcal{L}(X^\alpha)$ which is the restriction (or extension if $\alpha <0$) of $\{e^{-\mathcal{A}_\epsilon(t) s}: s\geqslant 0\}$ to $X^\alpha$. For more details we refer the reader to \cite{Amann, Triebel}. To shorten notation, we supress the index $\alpha$ and we write $\mathcal{A}_\epsilon(t)$ for all different realizations of this operator.

In this framework the problem \eqref{eq:plate} can be rewritten as an ordinary differential equation 
\begin{eqnarray}\label{eq:syst-nonl}
 \dfrac{d}{dt} \left[\begin{array}{c}
  u\\
  v \end{array}\right] + \mathcal{A}_\epsilon(t) \left[\begin{array}{c}
  u\\
  v \end{array}\right]   = F \left( \left[\begin{array}{c}
  u\\
  v \end{array}\right] \right),
\end{eqnarray}
with
$
F \left( \left[\begin{array}{c}
  u\\
  v \end{array}\right] \right) = \left[\begin{array}{c}
  0 \\
  f^e(u) \end{array}\right],
$
where $f^e$ is the Nemitski$\breve{\mbox{\i}}$ operator associated to $f$.

In order to obtain solutions of \eqref{eq:syst-nonl} we will need some information about the solution operator asso\-cia\-ted to the linear homogeneous problem
\begin{eqnarray}\label{eq:homog-gen}
 \dfrac{d}{dt} \left[\begin{array}{c}
  u\\
  v \end{array}\right] + \mathcal{A}_\epsilon(t) \left[\begin{array}{c}
  u\\
  v \end{array}\right]   = \left[\begin{array}{c}
  0 \\
  0 \end{array}\right], \qquad  \left[\begin{array}{c}
  u(t)\\
  v(t) \end{array}\right]_{t=t_0}= \left[\begin{array}{c}
  u_0\\
  v_0 \end{array}\right] \in X^{\alpha},
\end{eqnarray}
and to do this we introduce the following definitions:

\begin{definition}
Let $\mathcal{X}$ be a Banach space and assume that for all $t \in \R$ the linear operators $A(t): D \subset \mathcal{X} \to \mathcal{X}$ are closed and densely defined $($with $D$ independent of $t)$.

\begin{enumerate}
\item[a)] We say that $A(t)$ is uniformly sectorial $($in $\mathcal{X})$ if there is a constant $M>0$ $($in\-de\-pen\-dent of $t)$ such that
\begin{equation}
\left\| ({A}(t) + \mu I)^{-1} \right\|_{\mathcal{L}(\mathcal{X})} \leqslant \frac{M}{ |\mu| + 1}, \quad \forall \; \mu \in \C, \; {\rm Re}\, (\mu) \geqslant 0.
\end{equation}

\item[b)] We say that the map $t \mapsto{A}(t)$ is uniformly Hölder continuous $($in $\mathcal{X})$, if there are cons\-tants $C>0$ and $0<\beta<1$, such that for any $t,\tau,s \in \R$,

\begin{equation}
\left\| [{A}(t)-{A}(\tau)]{A}(s)^{-1} \right\|_{\mathcal{L}(\mathcal{X})} \leqslant C (t-\tau)^\beta.
\end{equation}

\item[c)] We say that a family of linear operators $\{S(t,\tau): t\geqslant \tau \in \R \} \subset \mathcal{L}(\mathcal{X})$ is a linear evolution process if
\begin{enumerate}
\item[1)] $S(\tau,\tau)=I$,

\item[2)]  $S(t, \sigma)S(\sigma,\tau)=S(t,\tau)$, for any $t \geqslant \sigma \geqslant \tau$,

\item[3)]  $(t,\tau) \mapsto S(t,\tau)v$ is continuous for all $t \geqslant \tau$ and $v \in \mathcal{X}$.
\end{enumerate}
\end{enumerate}

\end{definition}

\medskip

Notice that the requirements on $a_\epsilon$, $\epsilon \in [0,1]$ and the characterization of the resolvent operator
\[
\mathcal{A}_\epsilon(t)^{-1}=\left[\begin{array}{cc}
  (A+\lambda)^{-1}(A^\frac{1}{2} + a_\epsilon(t,\cdot)I)   & (A+\lambda)^{-1}\\
  -I  & 0
  \end{array}\right]
\] 
guarantee that the operators $\mathcal{A}_\epsilon$ are uniformly sectorial, and the map $t \mapsto \mathcal{A}_\epsilon(t)$ is uniformly H\"older continuous in $X^{0}$, uniformly in $\epsilon$. Therefore, following \cite{CN}, it is possible to construct a family $\{L_\epsilon(t,\tau): t\geqslant \tau \in \R \} \subset \mathcal{L}(X^{0})$ of linear evolution process that solves \eqref{eq:homog-gen}, for each $\epsilon \in [0,1]$.
  
  \medskip


\begin{definition} Let $F: X^\alpha \to X^{\beta}$, $\alpha \in [\beta, \beta+1)$, be a continuous function. We say that a continuous function $x:[t_0,t_0+{\tau}] \to X^\alpha $ is a $($local$)$ solution of  \eqref{eq:syst-nonl} starting in $ x_0 \in X^\alpha$, if $x \in C([t_0,t_0+{\tau}], X^\alpha) \cap C^1((t_0,t_0+{\tau}], X^{\alpha})$, $x(t_0)=x_0$, $x(t) \in D(\mathcal{A}_\epsilon(t))$ for all $t \in (t_0,t_0+{\tau}]$ and \eqref{eq:syst-nonl} is satisfied for all $t \in (t_0,t_0+{\tau})$.
\end{definition}

We can now state the following result, proved in \cite[Theorem 3.1]{CN}

\begin{theorem}\label{teo:existunicsol}
Suppose that the family of operators $\mathcal{A}(t)$ is uniformly sectorial and uniformly Hölder continuous in $X^{\beta}$. If $F:X^\alpha \to X^{\beta}$, $\alpha \in [\beta, \beta +1)$, is a Lipschitz continuous map in bounded subsets of $X^\alpha$, then, given $r > 0$, there is a time $\tau > 0$ such that for all $x_0 \in B_{X^\alpha }(0,r)$ there exists a unique solution of the problem \eqref{eq:syst-nonl} starting in $x_0$ and defined in $[t_0, t_0+\tau]$. Moreover, such solutions are continuous with respect the initial data in $B_{X^\alpha }(0,r)$.
\end{theorem}

Next we present the class of nonlinearities that we will consider.

\begin{lemma}\label{lem:fregul}
Let $f \in C^1(\R)$ be a function such that there exist constants $c>0$ and $\rho > 1$ such that  $|f'(s)| \leqslant c(1+|s|^{\rho-1})$, $\forall \, s\in \R$. Then  
$$| f(s)-f(t) |  \leqslant  2^{\rho-1}c \, | t-s | \big(1 +| s |^{\rho-1} + | t |^{\rho-1}  \big),$$ $\forall \, s,t \in \R$ .
\end{lemma}
\begin{proof}
For $a,b,s >0$, one has $(a+b)^s \leqslant 2^s \max\{a^s,b^s\} \leqslant 2^s(a^s + b^s)$. Hence,  given $s,t \in \R$, it follows from Mean Value's Theorem the existence of $\theta \in (0,1)$ such that
\[
\begin{split}
|f(s)-f(t)| & = |s-t| |f'\big( (1-\theta) s + \theta t \big)|\leqslant c |s-t| \, (1 + |(1-\theta) s + \theta t |^{\rho-1}) \\
&\leqslant  2^{\rho-1}c  |s-t|  \, (1 +  |(1-\theta) s|^{\rho-1} + |\theta t |^{\rho-1}) \leqslant  2^{\rho-1}c |s-t|  \, (1 +  | s|^{\rho-1} + | t |^{\rho-1}) .
\end{split}
\]
\end{proof}

\begin{lemma}\label{lem:fewelldef}
Assume that $1<\rho < \dfrac{n+4}{n-4}$ and let $f \in C^1(\R)$ be a function such that there exists a constant $c>0$ such that  $|f'(s)| \leqslant c(1+|s|^{\rho-1})$, $\forall \, s\in \R$. Then there exists $\alpha \in (0,1)$ such that the Nemitski$\breve{\mbox{\i}}$ operator $f^e: E^{\frac{1}{2}}\to E^{-\frac{\alpha}{2}}$ is Lipschitz continuous in bounded subsets of $ E^{\frac{1}{2}}$.
\end{lemma}
\begin{proof}
Let  be $\alpha \in (0,1)$ such that 
\begin{equation}\label{eq:alpha}
\rho \leqslant  \dfrac{n+4 \alpha}{n-4}.
\end{equation}
 Since $E^{\gamma} \hookrightarrow H^{{4 \gamma}}(\Omega)$, we have $E^{\frac{1}{2}} \hookrightarrow E^{\frac{\alpha}{2}} \hookrightarrow H^{{2 \alpha}}(\Omega) \hookrightarrow L^{\frac{2n}{n-4\alpha}}$. Therefore $L^{\frac{2n}{n+4\alpha}}(\Omega) \hookrightarrow  E^{-\frac{\alpha}{2}}$. Now by Lemma \ref{lem:fregul} and Hölder's Inequality we obtain
\[
\begin{split}
\|f^e(u) - f^e(v)\|_{E^{-\frac{\alpha}{2}}} & \leqslant \tilde{c}\,  \|f^e(u) - f^e(v)\|_{L^{\frac{2n}{n+4\alpha}}(\Omega)} \\
& \displaystyle \leqslant \tilde{c}\, \left( \int_{\Omega} \left[ 2^{\rho-1}c \, |u-v|(1 + |u|^{\rho-1} +  |v|^{\rho-1}) \right]^{\frac{2n}{n+4 \alpha}} \right)^{\frac{n+4 \alpha}{2n}} \\
& \displaystyle \leqslant \tilde{\tilde{c}}\,  \|u-v\|_{L^{\frac{2n}{n-4 \alpha}} (\Omega)}  \left( \int_{\Omega} \left( 1 +  |u|^{\rho-1} +  |v|^{\rho-1} \right)^{\frac{n}{4 \alpha}} \right)^{\frac{4 \alpha}{n}} \\
& \displaystyle \leqslant \tilde{\tilde{\tilde{c}}}\, \|u-v\|_{L^{\frac{2n}{n-4 \alpha}} (\Omega)} \left(1 + \|u\|_{L^{\frac{n(\rho-1)}{4 \alpha}}(\Omega)}^{\rho-1}  + \|v\|_{L^{\frac{n(\rho-1)}{4 \alpha}}(\Omega)}^{\rho-1} \right),
\end{split}
\]
where ${\tilde{c}}\,$ is the embedding constant from $L^{\frac{2n}{n+4 \alpha}}(\Omega)$ to $E^{-\frac{\alpha}{2}}$.

From Sobolev embeddings $E^{\frac{1}{2}} \hookrightarrow E^{\frac{\alpha}{2}} \hookrightarrow H^{2 \alpha}(\Omega) \hookrightarrow L^{\frac{n(\rho-1)}{4 \alpha}}(\Omega)$  for all $1< \rho \leqslant \dfrac{n+4 \alpha}{n-4}$, it follows that

\begin{eqnarray*}
\|f^e(u) - f^e(v)\|_{E^{-\frac{\alpha}{2}}} \leqslant C_1 \, \|u-v\|_{E^{\frac{1}{2}}} \left(1 +  \|u\|_{E^{\frac{1}{2}}}^{\rho-1} + \|v\|_{E^{\frac{1}{2}}}^{\rho-1} \right),
\end{eqnarray*}
for some constant $C_1>0$.
\end{proof}

\begin{remark}\label{Lips-function}
Since $L^{\frac{2n}{n+4}}(\Omega) \hookrightarrow L^2(\Omega)$, it follows from the proof of the Lemma \ref{lem:fewelldef} that $f^e: E^{\frac{1}{2}} \to L^2(\Omega)$ is Lipschitz continuous in bounded subsets, that is,
\[
\|f^e(u) - f^e(v)\|_{L^2(\Omega)}  \leqslant  \tilde{c}\,  \|f^e(u) - f^e(v)\|_{L^{\frac{2n}{n+4}}(\Omega)} \leqslant \tilde{\tilde{c}} \|u-v\|_{E^{\frac{1}{2}}}.
\]
\end{remark}

\begin{corollary}\label{corol:lips}
If $f$ is like in the Lemma $\ref{lem:fewelldef}$  and $\alpha \in (0,1)$ satisfies \eqref{eq:alpha}, the function $F: X^0 \to X^{-\alpha}$, given by $F \left( \left[\begin{array}{c}
  u\\
  v \end{array}\right] \right) = \left[\begin{array}{c}
  0 \\
  f^e(u) \end{array}\right]$, is Lipschitz continuous in bounded subsets of $ X^{0}$.
\end{corollary}

Now, Theorem \ref{teo:existunicsol} guarantees the local well posedness for the problem \ref{lem:fewelldef} in the energy space $H^2(\Omega) \times L^2(\Omega)$.

\begin{corollary}\label{corol:exst-loc}
If $f, F$  are like in the Corollary $\ref{corol:lips}$ and $\alpha \in (0,1)$ satisfies \eqref{eq:alpha}, then given $r >~0$, for each $\epsilon \in [0,1]$  there is a time $\tau=\tau(r) > 0$, such that for all $x_0 \in B_{X^{0}}(0,r)$ there exists a unique solution $x:~[t_0,t_0+{\tau}] \to X^{0} $ of the problem \eqref{eq:syst-nonl} starting in $x_0$. Moreover, such solutions are continuous with respect the initial data in $B_{X^0}(0,r)$.
\end{corollary}

Since $\tau$ can be chosen uniformly in bounded subsets of $X^0$, the solutions which do not blow up in $X^0$ must exist globally.


\section{Existence of global solution}

In this section we establish estimates in $X^0$ which implies global existence of solutions of the problem \eqref{eq:syst-nonl}. The choice of $X^{0}$ is suitable to study the asymptotic behaviour of \eqref{eq:plate}, since we may exhibit an energy functional in this space.

We consider the norms
$$
\|u\|_{\frac{1}{2}}:= \left[ \|\Delta u\|_{L^2(\Omega)}^2 + \lambda \| u \|_{L^2(\Omega)}^2 \right ]^{\frac{1}{2}}
$$
and
$$
\left\|\left[\begin{array}{c}
  u\\
  v \end{array}\right] \right\|_{X^0}=\left[\|u\|_{\frac{1}{2}}^2 +  \|v\|_{L^2(\Omega)}^2 \right]^{\frac{1}{2}},
$$
which are equivalent to the usual ones in $E^\frac{1}{2}= H^2(\Omega)\cap H^1_0(\Omega)$ and $X^0 = H^2(\Omega)\cap H^1_0(\Omega) \times L^2(\Omega)$, respectively. 

For any $0<b \leqslant \dfrac{1}{4}$, usind Young's and Cauchy-Schwarz Inequality, we obtain

\begin{align}\label{eq:pro-int-equi}
- \dfrac{1}{4} \left[  \| u \|_{\frac{1}{2}}^2 +  \| v \|_{L^2(\Omega)}^2 \right]   & \leqslant - b \left[ \lambda \| u \|_{L^2(\Omega)}^2 +  \| v \|_{L^2(\Omega)}^2 \right] \leqslant 2b \lambda^{\frac{1}{2}}  \left< u, v \right>_{L^2(\Omega)}  \\
& \leqslant b \left[  \lambda \| u \|_{L^2(\Omega)}^2 +  \| v \|_{L^2(\Omega)}^2 \right]  \leqslant  \dfrac{1}{4} \left[  \| u \|_{\frac{1}{2}}^2 +  \| v \|_{L^2(\Omega)}^2 \right] , \nonumber
\end{align}
which leads to 
\begin{equation}\label{eq:eqen}
\dfrac{1}{4}  \left\|\left[\begin{array}{c}
  u\\
  v \end{array}\right] \right\|_{X^{0}}^2  \leqslant \dfrac{1}{2}  \left\|\left[\begin{array}{c}
  u\\
  v \end{array}\right] \right\|_{X^{0}}^2 + 2b \lambda^{\frac{1}{2}} \left< u, v \right>_{L^2(\Omega)} \leqslant \dfrac{3}{4}  \left\|\left[\begin{array}{c}
  u\\
  v \end{array}\right] \right\|_{X^{0}}^2 .
\end{equation}

First of all, we deal with the homogeneous problem \eqref{eq:homog-gen}. In this case we define the functional $W:X^{0} \to \R$ by
\begin{equation} \label{eq:enlin}
W\left( \left[\begin{array}{c}
  u\\
  v \end{array}\right] \right)  =   \frac{1}{2}  \left\|\left[\begin{array}{c}
  u\\
  v \end{array}\right] \right\|_{X^{0}}^2 + 2b \lambda^{\frac{1}{2}} \left< u,v \right>_{L^2(\Omega)} .
\end{equation}

\bigskip

If $x=\left[\begin{array}{c}
  u\\
  v \end{array}\right] : [t_0,t_0 +\tau] \to X^0$ is the solution of the problem \eqref{eq:homog-gen} starting in $x_0 = \left[\begin{array}{c}
  u_0\\
  v_0 \end{array}\right] \in X^0$, then $u=u(t)$ is a solution (local in time) of the homogeneous problem
\begin{equation}\label{eq:new}
\begin{cases}
u_{tt} + a_\epsilon(t,x)u_{t} +(- \Delta u_{t}) +  (-\Delta)^{2} u + \lambda u = 0 & \qquad \hbox{ in } \Omega, \\
u=\Delta u = 0 & \qquad \hbox{ on }  \partial \Omega.
\end{cases}
\end{equation}

Putting $v=u_{t}$ in \eqref{eq:enlin}, we have by regularity of $u$ given in the Corollary \ref{corol:exst-loc} and by Young's Inequality that
\[\begin{split}
\frac{d}{dt} & W \left(\left[\begin{array}{c}
  u\\
  u_t \end{array}\right] \right)   \\
& =    \left< \Delta u, \Delta u_{t} \right>_{L^2(\Omega)} + \lambda  \left< u,u_{t} \right>_{L^2(\Omega)} +    \left< u_{t},u_{tt}  \right>_{L^2(\Omega)}  + 2b\lambda^{\frac{1}{2}} \left< u_{t},u_{t} \right>_{L^2(\Omega)}+ 2b\lambda^{\frac{1}{2}} \left< u,u_{tt} \right>_{L^2(\Omega)} \\
& =  \left< \Delta u, \Delta u_{t} \right>_{L^2(\Omega)} + \lambda  \left< u,u_{t} \right>_{L^2(\Omega)} +  \left< u_{t}, -a_\epsilon(t,x)u_{t} - (-\Delta)^2 u-(-\Delta)u_t - \lambda u \right>_{L^2(\Omega)} \\
& \quad + 2b\lambda^{\frac{1}{2}} \left< u_{t},u_{t} \right>_{L^2(\Omega)} + 2b\lambda^{\frac{1}{2}} \left< u, -a_\epsilon(t,x)u_{t} - (-\Delta)^2 u - (-\Delta)u_t - \lambda u \right>_{L^2(\Omega)} \\
& \leqslant - (\alpha_{0} - 2b\lambda^{\frac{1}{2}})\|u_{t}\|^2_{L^2(\Omega)} + 2b \alpha_1 \lambda^{\frac{1}{2}} \,\left< -u,u_{t} \right>_{L^2(\Omega)} - 2b\lambda^{\frac{1}{2}} \left< u, (-\Delta)^2 u \right>_{L^2(\Omega)} \\
& \quad - 2b\lambda^{\frac{1}{2}} \left<u,-\Delta u_t \right>_{L^2(\Omega)}  - 2b \lambda^{\frac{3}{2}} \|u\|^2_{L^2(\Omega)} \\
& \leqslant - (\alpha_{0} - 2b\lambda^{\frac{1}{2}} - b \lambda^{\frac{1}{2}})\|u_{t}\|^2_{L^2(\Omega)} + 2b \alpha_{1}\lambda^{\frac{1}{2}} \|u\|_{L^2(\Omega)} \|u_{t}\|_{L^2(\Omega)} -( 2b\lambda^{\frac{1}{2}}-b \lambda^{\frac{1}{2}}) \|\Delta u\|^2_{L^2(\Omega)}\\
& \quad - 2b \lambda^{\frac{3}{2}} \|u\|^2_{L^2(\Omega)} \\
& \leqslant - (\alpha_{0} - 2b\lambda^{\frac{1}{2}} -b \lambda^{\frac{1}{2}})\|u_{t}\|^2_{L^2(\Omega)} + \dfrac{b \alpha_{1} \lambda^{\frac{1}{2}} }{\eta} \|u_{t}\|^2_{L^2(\Omega)} + b \alpha_{1} \lambda^{\frac{1}{2}} \eta  \|u\|^2_{L^2(\Omega)}  \\
& \quad - (2 b\lambda^{\frac{1}{2}}-b \lambda^{\frac{1}{2}})  \|\Delta u\|^2_{L^2(\Omega)} - 2b \lambda^{\frac{3}{2}}  \|u\|^2_{L^2(\Omega)} \\
& \leqslant -(\alpha_{0} - 2b\lambda^{\frac{1}{2}} - b \lambda^{\frac{1}{2}} - \dfrac{b \alpha_{1} \lambda^{\frac{1}{2}} }{\eta} )\|u_{t}\|^2_{L^2(\Omega)} + \lambda^{\frac{1}{2}} (b \alpha_{1}\eta -b\lambda )\|u\|^2_{L^2(\Omega)} \\
& \quad - b\lambda^{\frac{1}{2}} ( \|\Delta u\|^2_{L^2(\Omega)} + \lambda  \|u\|^2_{L^2(\Omega)}),
\end{split}\]
for all $\eta >0$. Taking $\eta= \dfrac{\lambda}{\alpha_{1}}$ it follows that
$$
\frac{d}{dt}W\left( \left[\begin{array}{c}
  u\\
  u_t \end{array}\right] \right) \leqslant  -(\alpha_{0} - 2b\lambda^{\frac{1}{2}} - b \lambda^{\frac{1}{2}} - \dfrac{b \alpha_{1}^2} {\lambda^{\frac{1}{2}} } ) \|u_{t}\|^2_{L^2(\Omega)} - b \lambda^{\frac{1}{2}} ( \|\Delta u\|^2_{L^2(\Omega)} + \lambda  \|u\|^2_{L^2(\Omega)}).$$

Choosing $0< b\leqslant \dfrac{1}{4}$ such that $\alpha_{0} - 2b\lambda^{\frac{1}{2}} - b \epsilon \lambda^{\frac{1}{2}} - \dfrac{b \alpha_{1}^2} {\lambda^{\frac{1}{2}} } > 0$, and taking $\delta = \min\{\alpha_{0} - 2b\lambda^{\frac{1}{2}} - b \lambda^{\frac{1}{2}} - \dfrac{b \alpha_{1}^2} {\lambda^{\frac{1}{2}} }, b \lambda^{\frac{1}{2}}  \}>0$ we have that \eqref{eq:eqen} implies that
$$
\frac{d}{dt}W\left( \left[\begin{array}{c}
  u\\
  u_t \end{array}\right] \right)  \leqslant -\delta  \left[ \| u \|_{\frac{1}{2}}^2 +  \| u_{t} \|_{L^2(\Omega)}^2 \right] \leqslant -\frac{4 \delta}{3} \, W\left( \left[\begin{array}{c}
  u\\
  u_t \end{array}\right] \right).
$$

Therefore
$$
\dfrac{1}{4} \left\| x(t) \right\|_{X^{0}}^2 \leqslant W\left( \left[\begin{array}{c}
  u_0\\
  v_0 \end{array}\right] \right) \, e^{-\frac{4 \delta}{3} (t-t_0)} \leqslant 3 \left\|  \left[\begin{array}{c}
  u_0\\
  v_0 \end{array}\right] \right\|_{X^{0}}^2 e^{-\frac{4 \delta}{3} (t-t_0)},
  $$
for all $t\in [t_0, t_0+\tau]$.

Hence we conclude that the solutions of \eqref{eq:new} are uniformly exponentially dominated for initial data $x_0$ in bounded subsets $B \subset X^{0}$.

In order to get energy estimates in the semilinear case \eqref{eq:syst-nonl}, we assume besides of the hypothesis in the Corollary \ref{corol:lips}, the dissipativeness condition
\begin{equation}\label{dissipative-condition}
\limsup_{|s|\to \infty} \frac{f(s)}{s} \leqslant 0.
\end{equation}

In this case we consider the following functional $\mathcal{W}:X^0 \to \R$
\begin{equation}\label{eq:ener-non-lin}
\mathcal{W}\left( \left[\begin{array}{c}
  u\\
  v \end{array}\right] \right)  =   W\left( \left[\begin{array}{c}
  u\\
  v \end{array}\right] \right)  - \int_{\Omega} \left[\begin{array}{c}
  0\\
  \mathcal{F}^e(u) \end{array}\right] \, dx,
\end{equation}
where $\mathcal{F}^e$ is the Nemitski$\breve{\mbox{\i}}$ map associated to a primitive of $f$, $\displaystyle \mathcal{F}(s)=\int_0^s f(t)\, dt$.

Now we suppose that $x=\left[\begin{array}{c}
  u\\
  v \end{array}\right] : [t_0,t_0 +\tau] \to X^0$ is the solution of the problem \eqref{eq:syst-nonl} starting in $x_0=\left[\begin{array}{c}
  u_0\\
  v_0 \end{array}\right] \in X^0$. Therefore $u=u(t)$ is a solution (local in time) of the equation
\begin{equation*}
\begin{cases}
u_{tt} + a_\epsilon(t,x)u_{t} + (- \Delta u_{t} )+  (-\Delta)^{2} u + \lambda u = f(u) & \qquad \hbox{ in } \Omega, \\
u=\Delta u = 0 & \qquad \hbox{ on }  \partial \Omega .
\end{cases}
\end{equation*}
Similarly to the homogeneous case we have
\[
\begin{split}
&\frac{d}{dt} \mathcal{W}\left( \left[\begin{array}{c}
  u\\
  u_t \end{array}\right] \right)   \\
& =    \left< \Delta u, \Delta u_{t} \right>_{L^2(\Omega)} + \lambda  \left< u,u_{t} \right>_{L^2(\Omega)} +    \left< u_{t},u_{tt}  \right>_{L^2(\Omega)}  + 2b\lambda^{\frac{1}{2}} \left< u_{t},u_{t} \right>_{L^2(\Omega)}+ 2b\lambda^{\frac{1}{2}} \left< u,u_{tt} \right>_{L^2(\Omega)}\\
& \quad- \int_{\Omega} f(u) u_t dx \\
& = \left< \Delta u, \Delta u_{t} \right>_{L^2(\Omega)} + \lambda  \left< u,u_{t} \right>_{L^2(\Omega)} +  \left< u_{t}, -a_\epsilon(t,x)u_{t} - (-\Delta)^2 u-(-\Delta)u_t - \lambda u + f(u) \right>_{L^2(\Omega)}  \\
&\quad +  2b\lambda^{\frac{1}{2}} \left< u_{t},u_{t} \right>_{L^2(\Omega)} + 2b\lambda^{\frac{1}{2}} \left< u, -a_\epsilon(t,x)u_{t} - (-\Delta)^2 u-(-\Delta)u_t - \lambda u + f(u) \right>_{L^2(\Omega)} \\
& \qquad - \int_{\Omega} f(u) u_t dx \\
&  \leqslant   -(\alpha_{0} - 2b\lambda^{\frac{1}{2}} - b \lambda^{\frac{1}{2}} - \dfrac{b \alpha_{1} \lambda^{\frac{1}{2}} }{\eta} )\|u_{t}\|^2_{L^2(\Omega)} + \lambda^{\frac{1}{2}} (b \alpha_{1}\eta -b\lambda )\|u\|^2_{L^2(\Omega)} \\
& \quad - b \lambda^{\frac{1}{2}} ( \|\Delta u\|^2_{L^2(\Omega)} + \lambda  \|u\|^2_{L^2(\Omega)}) +  2b \lambda^{\frac{1}{2}} \int_{\Omega}f(u)udx.
\end{split}
\]
for all $\eta >0$.

To deal with the integral term, just notice that from dissipativeness condition \eqref{dissipative-condition}, for all $\nu >0$ given, there exists $R_\nu >0$ such that for $|s| > R_ \nu$ one has $f(s)s \leqslant \nu  s^2$. Moreover being the function $f(s)s$ bounded in the interval $|s| \leqslant R_ \nu $ there exists a constant $M_\nu $ such that $f(s)s \leqslant M_\nu + \nu s^2$ for all $s \in \R$.

Therefore, given $\nu >0$ there exists $C_\nu >0$ such that
\[
\int_{\Omega} f(u) u \,dx \leqslant \nu \|u\|_{L^2(\Omega)}^{2} + C_{\nu}.
\]

Therefore
\[
\begin{split}
\frac{d}{dt}  \mathcal{W}\left( \left[\begin{array}{c}
  u\\
  u_t \end{array}\right] \right)& \leqslant  -(\alpha_{0} - 2b\lambda^{\frac{1}{2}} - b \lambda^{\frac{1}{2}} - \dfrac{b \alpha_{1} \lambda^{\frac{1}{2}} }{\eta} )\|u_{t}\|^2_{L^2(\Omega)} + \lambda^{\frac{1}{2}} (b \alpha_{1}\eta -b\lambda )\|u\|^2_{L^2(\Omega)} \\
& \quad - b \lambda^{\frac{1}{2}} ( \|\Delta u\|^2_{L^2(\Omega)} + \lambda  \|u\|^2_{L^2(\Omega)})+ 2b \lambda^{\frac{1}{2}}(\nu \|u\|_{L^2(\Omega)}^{2} + C_{\nu})  \\
&\leqslant  -(\alpha_{0} - 2b\lambda^{\frac{1}{2}} - b \lambda^{\frac{1}{2}} - \dfrac{b \alpha_{1} \lambda^{\frac{1}{2}} }{\eta} )\|u_{t}\|^2_{L^2(\Omega)} + \lambda^{\frac{1}{2}} (b \alpha_{1}\eta -b\lambda +2b\nu)\|u\|^2_{L^2(\Omega)} \\
& \quad - b \lambda^{\frac{1}{2}} ( \|\Delta u\|^2_{L^2(\Omega)} + \lambda  \|u\|^2_{L^2(\Omega)}) + 2b \lambda^{\frac{1}{2}} C_{\nu}.
\end{split}
\]
Now, fixing $\displaystyle \nu \in (0, \frac{\lambda}{2})$ and taking $\displaystyle \eta = \frac{\lambda - 2\nu}{\alpha_1}>0$, we have
\[
\frac{d}{dt}\mathcal{W}\left( \left[\begin{array}{c}
  \!u\! \\
  \!u_t\! \end{array}\right] \right)  \leqslant   -(\alpha_{0} - 2b\lambda^{\frac{1}{2}} - b \lambda^{\frac{1}{2}} - \dfrac{b \alpha_{1} \lambda^{\frac{1}{2}} }{\eta} ) \|u_{t}\|^2_{L^2(\Omega)} - b \lambda^{\frac{1}{2}} ( \|\Delta u\|^2_{L^2(\Omega)} + \lambda  \|u\|^2_{L^2(\Omega)}) + 2b \lambda^{\frac{1}{2}} C_{\nu}.
\]
Choosing $\displaystyle 0< b < \frac{\alpha_0}{\lambda^{\frac{1}{2}}(2\eta +\eta + \alpha_1)}$ and $\displaystyle \omega = \min \{ \alpha_{0} - 2b\lambda^{\frac{1}{2}} - b \lambda^{\frac{1}{2}} - \dfrac{b \alpha_{1} \lambda^{\frac{1}{2}} }{\eta} , b\lambda^{\frac{1}{2}} \} > 0$, we have
\[
\frac{d}{dt} \mathcal{W}\left( \left[\begin{array}{c}
  u\\
  u_t \end{array}\right] \right) \leqslant  - \omega \left\|\left[\begin{array}{c}
  u\\
  u_t \end{array}\right]  \right\|_{X^{0}}^2+ 2b \lambda^{\frac{1}{2}}   C_{\nu}.
\]

Now we observe that if $\xi \in H^2(\Omega) \hookrightarrow L^{\frac{2n}{n-4}}(\Omega) $, then $|\xi|^{\rho+1} \in L^{\frac{2n}{(n-4)(\rho+1)}}(\Omega) \hookrightarrow L^1(\Omega)$ for all $1 < \rho < \dfrac{n+4}{n-4}$, and our hypothesis on $f$ implies that $|f(s)| \leqslant  c(1 + |s|^{\rho}),\, s\in \mathbb{R}$.

Therefore we can find a constant $\bar{c}>1$ such that for all $\xi \in E^\frac{1}{2}$
\[
-\int_{\Omega}\int_{0}^{\xi(x)} f(s)dsdx \leqslant \bar{c} \| \xi\|_{\frac{1}{2}}^{2} (1+ \| \xi\|_{\frac{1}{2}}^{\rho -1}),
\]
and therefore
\begin{equation}\label{sem-ideia}
-d\int_{\Omega}\int_{0}^{\xi(x)}f(s) ds dx \leqslant \| \xi\|_{\frac{1}{2}}^{2},
\end{equation}
whenever $\| \xi\|_{\frac{1}{2}} \leqslant r$ and considering $\displaystyle d= \frac{1}{\bar{c}(1+ r^{\rho-1})}<1$.

Hence from \eqref{sem-ideia} we derive that
\[
\begin{split}
- \frac{\omega}{2}  \left\|\left[\begin{array}{c}
  u\\
  u_t \end{array}\right]  \right\|_{X^{0}}^2 & = -\frac{\omega}{2} \|u\|_{\frac{1}{2}}^{2} -\frac{\omega}{2}\|u_t\|_{L^{2}(\Omega)}^{2} \leqslant  -\frac{\omega}{2} \|u\|_{\frac{1}{2}}^{2}\
 \leqslant \frac{\omega d}{2}\int_{\Omega} \int_{0}^{u} f(s)ds dx \\
\end{split}
\]
and we obtain
\[
\begin{split}
\frac{d}{dt} \mathcal{W}\left( \left[\begin{array}{c}
  u\\
  u_t \end{array}\right] \right)  &\leqslant  - \frac{\omega}{2}  \left\| \left[\begin{array}{c}
  u\\
  u_t \end{array}\right] \right\|_{X^{0}}^2 + \frac{d\omega}{2}\int_{\Omega}\int_{0}^{u}f(s)ds dx + 2b \lambda^{\frac{1}{2}} {{C}}_{\nu} \\
&\leqslant - \frac{\omega}{2}\left[ 4\, W\left(\left[\begin{array}{c}
  u\\
  u_t \end{array}\right]  \right) + d \int_{\Omega}\int_{0}^{u}f(s)ds dx \right] + 2b \lambda^{\frac{1}{2}} {{C}}_{\nu}\\
& \leqslant - \bar{\omega}\mathcal{W} \left( \left[\begin{array}{c}
  u\\
  u_t \end{array}\right]  \right) +2b \lambda^{\frac{1}{2}} {C}_\nu
\end{split}
\]
where $\bar{\omega} = \min\{{2 \omega},\dfrac{d \omega}{2}\}$.

As in the homogeneous case, we conclude that this solutions are uniformly exponentially do\-mi\-nated for initial data $x_0$ in bounded subsets $B \subset X^{0}$, ie, there exist constants $K=K(B)$ and $K_1>0$ such that
\begin{equation}\label{decaimento-exponencial-naolinear}
\left \|x(t) \right\|_{X^{0}}^{2} \leqslant K e^{-\bar{\omega}(t-t_0)} + K_1,
\end{equation}
for all solution $x:[t_0,t_0 + \tau] \to X^0$ of the equation \eqref{eq:syst-nonl} starting in $x_0 \in B$.

\begin{remark}
Estimate \eqref{decaimento-exponencial-naolinear} and Corollary \ref{corol:exst-loc} allow us to consider for each initial data $x_0 \in X^0$ and each initial time $\tau \in \R$, the global solution $x_\epsilon=x_\epsilon(\cdot,\tau, x_0):[\tau,\infty) \to X^0$ of the equation \eqref{eq:syst-nonl} starting in $x_0$. This arises an evolution process $\{ S_\epsilon(t,\tau) : t\geqslant \tau \}$ in the state space $X^0$ defined by $S_\epsilon(t,\tau)x_0 = x_\epsilon(t,\tau,x_0)$. According to \cite{CN}
\begin{equation}\label{eq:evoper}
S_\epsilon(t,\tau)x_0 = L_\epsilon(t,\tau)x_0 + \int_\tau^t L_\epsilon(t,s)F(S_\epsilon(s,\tau)x_0)\, ds, \quad \forall \, t \geqslant \tau \in \R,
\end{equation}
where $\{L_\epsilon(t,\tau): t\geqslant \tau \in \R \}$ is the linear evolution process associated to the homogeneous problem \eqref{eq:homog-gen}.
\end{remark}


\section{Existence of pullback attractors}

In this section we prove the existence of pullback attractors for the problem \eqref{eq:plate} and  the upper-semi\-con\-tinuity of the family of pullback attractors when the parameter $\epsilon$ goes to $0$. For the sake of completness we will present basic definitions and results of the theory of pullback attractors. For more details the reader is invited to look  \cite{ChV,TLR,CCLR}.

We start remembering the definition of Hausdorff semi-distance between two subsets $A$ and $B$ of a metric space $(X,d)$:
\[
{\rm dist}_H(A,B) = \sup_{a\in A} \inf_{b\in B} d(a,b).
\]

\begin{definition}\label{pull attraction}
Let $\{S(t,\tau):\ t\geqslant \tau\in {\mathbb R}\}$ be an evolution
process in a metric space $X$. Given  $A$ and $B$ subsets of $X$, we
say that $A$ \emph{pullback attracts} $B$ at time $t$ if
$$
\lim_{\tau \rightarrow -\infty} {\rm dist}_H(S(t,\tau)B,A)= 0,
$$
where $S(t,\tau)B:= \{S(t,\tau)x \in X : x \in B\}$.
\end{definition}

\begin{definition}
The pullback orbit of a subset $B \subset X$ relatively to the evolution process $\{S(t,\tau):\ t\geqslant \tau\in {\mathbb R}\}$ in the time $t \in \R$ is defined by $\gamma_p(B,t):= \bigcup_{\tau \leqslant t} S(t,\tau)B $.
\end{definition}

\begin{definition}\label{pull-stro-bounded}
An evolution process $\{S(t,\tau): t \geqslant \tau\}$ in $X$ is pullback strongly bounded if, for each $t\in \mathbb{R}$ and each bounded subset $B$ of $X$, $\bigcup_{\tau\leqslant t} \gamma_p(B,\tau)$ is bounded.
\end{definition}


\begin{definition}\label{asym-comp}
An evolution process $\{S(t,\tau):\ t\geqslant \tau\in {\mathbb
R}\}$ in $X$ is pullback asymptotically compact if, for each $t \in \R$, each sequence $\{\tau_n\}$ in $(-\infty, t]$ with $\tau_n \stackrel{n\to\infty}{\longrightarrow}-\infty$ and each bounded sequence $\{x_n\}$ in $X$ such that $\{S(t,\tau_n)x_n\} \subset X $ is bounded, the sequence $\{S(t,\tau_n)x_n\}$ is relatively compact in $X$.
\end{definition}

\begin{definition}\label{pull-absor}
We say that a family of bounded subsets $\{B(t): t\in \R\}$ of
$X$ is pullback absorbing for the evolution process
$\{S(t,\tau):t\geqslant \tau\in \R\}$, if for each $t\in \R$ and for
any bounded subset $B$ of $X$, there exists $\tau_0(t,B)\leqslant t$ such
that
$$S(t,\tau)B\subset B(t) \quad \mbox{ for all  }  \; \tau \leqslant \tau_0(t,B).$$
\end{definition}

\begin{definition}
We say that a family of subsets $\{\mathbb{A}(t):t\in \R\}$ of $X$ is
invariant relatively to the evolution process
$\{S(t,\tau):t\geqslant \tau\in \R\}$ if $S(t,\tau) \mathbb{A}(\tau) = \mathbb{A}(t)$,
for any $t\geqslant \tau$.
\end{definition}

\begin{definition}\label{pull-attractor}
A family of subsets $\{{\mathbb A}(t):t\in \R\}$ of $X$ is
called a \emph{pullback attractor} for the evolution process $\{S(t,\tau):\ t\geqslant
\tau\in {\mathbb R}\}$ if it is invariant, $\mathbb{A}(t)$ is
compact for all $t\in \R$, and pullback attracts bounded subsets of
$X$ at time $t$, for each $t\in \R$.
\end{definition}

In the applications, to prove that a process has a pullback attractor we use the Theorem \ref{theorem pullback}, proved in \cite{CCLR}, which gives a sufficient condition for existence of a compact pullback attractor. For this, we will need the concept of pullback strongly bounded dissipativeness.

\begin{definition}
An evolution process $\{S(t,\tau):t\geqslant \tau \in \mathbb{R}\}$ in $X$ is \emph{pullback strongly bounded dissipative} if, for each $t\in \mathbb{R}$, there is a bounded subset $B(t)$ of $X$ which pullback absorbs bounded subsets of $X$ at time $s$ for each $s\leqslant t$; that is, given a bounded subset $B$ of $X$ and $s\leqslant t$, there exists $\tau_0(s,B)$ such that $S(s,\tau)B \subset B(t)$, for all $\tau\leqslant \tau_0(s,B)$.
\end{definition}

Now we can present the result which guarantees the existence of pullback attractors for non\-autonomous problems.

\begin{theorem}[\cite{CCLR}]\label{theorem pullback}
If an evolution process $\{S(t,\tau):\ t\geqslant \tau\in {\mathbb R}\}$ in the metric space $X$ is pullback strongly bounded dissipative and pullback asymptotically compact, then $\{S(t,\tau):\ t\geqslant \tau\in {\mathbb R}\}$ has a pullback attractor $\{\mathbb{A}(t): t\in \mathbb{R}\}$ with the property that $\bigcup_{\tau \leqslant t } \mathbb{A}(\tau)$ is bounded for each $t\in \mathbb{R}$.
\end{theorem}

Next result gives sufficient conditions for pullback asymptotic compactness, and its proof can be found in \cite{CCLR}.

\begin{theorem}[\cite{CCLR}]\label{pull-asym-comp}
Let $\{S(t,s): t\geqslant s\}$ be a pullback strongly bounded evolution process such that $S(t,s) = T(t,s) + U(t,s)$, where $U(t,s)$ is compact and there exist a non-increasing function $k: \mathbb{R}^{+} \times \mathbb{R}^{+} \to \mathbb{R}$, with $k(\sigma,r)\to 0$ when $\sigma \to \infty$, and for all $s\leqslant t$ and $x\in X$ with $\|x\| \leqslant r$, $\|T(t,s)x\| \leqslant k(t-s,r)$. Then, the family of evolution process $\{S(t,s): t\geqslant s\}$ is pullback asymptotically compact.
\end{theorem}

\begin{theorem}\label{teo-compact}
Considering in $X^0$, the family of operators
\[
\displaystyle U_\epsilon(t,\tau)(\cdot):= \int_\tau^t L_\epsilon(t,s)F(S_\epsilon(s,\tau) \cdot)\, ds,
\]
obtained from \eqref{eq:evoper}, the family of evolution process $\{U_\epsilon(t, \tau) :  t \geqslant \tau\}$ is compact in $X^{0}$.
\end{theorem}
\begin{proof}
The compactness of $U_\epsilon$ follows easily from the fact that
  $$
  E^\frac{1}{2} \stackrel{f^e}{\longrightarrow} X^{-\frac{\alpha}{2}} \hookrightarrow E^{-\frac{1}{2}},
  $$
being the last inclusion compact, since that $\alpha <1$.
\end{proof}

From estimate \eqref{decaimento-exponencial-naolinear} it is easy to check that the evolution process $\{S(t, \tau): t\geqslant \tau\}$ associated to  the equation \eqref{eq:syst-nonl} is pullback strongly bounded.
Hence, applying Theorem \ref{pull-asym-comp}, we obtain that the family of evolution process $\{S_\epsilon(t, \tau): t\geqslant \tau\}$ is pullback asymptotically compact. Now, applying Theorem \ref{theorem pullback} we get that equation \eqref{eq:plate} has a pullback attractor $\{\mathbb{A}_\epsilon(s): s\in \R\}$ in $X^{0}= H^2(\Omega) \cap H^1_0(\Omega) \times L^2(\Omega)$ and that $\bigcup_{s\in \R} \mathbb{A}_\epsilon(s) \subset X^0$ is bounded.


\subsection{Upper-semicontinuity of pullback attractors}

For each value of the parameter $\epsilon \in [0,1]$ we recall that $S_\epsilon(t,\tau)$ is the evolution process associated to semilinear problem \eqref{eq:syst-nonl}. Now we prove that the family of pullback attractors $\{\mathbb{A}_\epsilon(t)\}$ is upper-semicontinuous in $\epsilon=0$, ie, we show that
$$
\lim_{\epsilon \to 0} {\rm dist}(\mathbb{A}_\epsilon(t), \mathbb{A}_0(t))=0.
$$

Let be $Z\left( \left[\begin{array}{c}
  u\\
  v \end{array}\right] \right)  = \dfrac{1}{2} \left( \|u\|^2_{\frac{1}{2}} + \|v\|^2_{L^2(\Omega)} \right)$. For each $x_0 \in X^0$ consider $u=S_\epsilon(t,\tau)u_0$ and $v=S_0(t,\tau)u_0$. Let $w=u-v$. Then
\begin{equation}
w_{tt} =  a_0(t,x)v_t - a_\epsilon(t,x)u_t + \Delta w_t - \Delta^2 w - \lambda w + f(u) - f(v)
\end{equation}

It follows from Remark \ref{Lips-function} that $f$ is Lipschitz continuous in bounded set from $E^\frac{1}{2}$ to $L^2(\Omega)$. Since $u, v, u_t$ and $v_t$ are bounded, Young's Inequality leads to
\[
\begin{split}
\dfrac{d}{dt} &Z\left( \left[\begin{array}{c}
  w\\
  w_t \end{array}\right] \right) \\  & =  \langle w, w_t \rangle_{E^\frac{1}{2}} +  \langle w_t, w_{tt} \rangle_{L^2(\Omega)} \\
 & =   \langle \Delta w, \Delta w_{t} \rangle_{L^2(\Omega)} + \lambda \langle w, w_{t} \rangle_{L^2(\Omega)} +  \langle w_t, w_{tt} \rangle_{L^2(\Omega)} \\
 & =   \langle \Delta^2 w + \lambda w + w_{tt}, w_t \rangle_{L^2(\Omega)}  \\
 & =  \langle a_0(t,x)v_t - a_\epsilon(t,x)u_t  + \Delta w_t + f(u)-f(v), w_t \rangle_{L^2(\Omega)} \\
  & =   \langle - a_0(t,x)w_t + (a_0(t,x)-a_\epsilon(t,x))u_t, w_t \rangle_{L^2(\Omega)} - \|\nabla w_t \|^2_{L^2(\Omega)}  + \langle f(u)-f(v), w_t \rangle_{L^2(\Omega)} \\
 & \leqslant -\alpha_0 \| w_t \|^2_{L^2(\Omega)} + \|a_0 - a_\epsilon \|_{L^\infty(\R \times \Omega)} \|u_t \|_{L^2(\Omega)}  \| w_t \|_{L^2(\Omega)} + K(\|w\|^2_{L^2(\Omega)} + \|w_t\|^2_{L^2(\Omega)})\\
& \leqslant \tilde{K} Z\left( \left[\begin{array}{c}
  w\\
  w_t \end{array}\right] \right) + \tilde{K}  \|a_0 - a_\epsilon \|_{L^\infty(\R \times \Omega)}.
\end{split}
\]

Therefore,
\[
\begin{split}
 Z\left( \left[\begin{array}{c}
  w(t)\\
  w_t (t) \end{array}\right] \right) & \leqslant \tilde{K} \int_\tau^t  Z\left( \left[\begin{array}{c}
  w(s)\\
  w_s (s) \end{array}\right] \right) ds + \tilde{K}(t-\tau)\|a_0 - a_\epsilon \|_{L^\infty(\R \times \Omega)} + Z\left( \left[\begin{array}{c}
  w(\tau)\\
  w_t (\tau) \end{array}\right] \right) \\
  & \leqslant \tilde{\tilde{K}}  \int_\tau^t  Z\left( \left[\begin{array}{c}
  w(s)\\
  w_s (s) \end{array}\right] \right) ds + \tilde{\tilde{K}}  (t-\tau)\|a_0 - a_\epsilon \|_{L^\infty(\R \times \Omega)}
  \end{split}
\]
where $\tilde{\tilde{K}} = \max\left\{\tilde{{K}} , \dfrac{Z\left( \left[\begin{array}{c}
  w(\tau)\\
  w_t (\tau) \end{array}\right] \right)}{(\alpha_1 - \alpha_0)} \right\}.$

Hence, by Gronwall's Inequality it follows that
\begin{equation}\label{eq:grownep}
\|w\|^2_{\frac{1}{2}} + \|w_t\|^2_{L^2(\Omega)}  \leqslant \tilde{\tilde{\tilde{K}}} \|a_0 - a_\epsilon \|_{L^\infty(\R \times \Omega)}  \int_\tau^t e^{K(t-s)} \, ds \to 0,
\end{equation}
as $\epsilon \to 0$ in compact subsets of $\R$ uniformly for $x_0$ in bounded subsets of $X^0$.

For $\delta >0$ given, let $\tau \in \R$ be such that ${\rm dist}(S_0(t,\tau)B,\mathcal{A}_0(t)) < \frac{\delta}{2}$, where $\displaystyle B \supset \bigcup_{s \in \R}\mathcal{A}_\epsilon(s)$ is a bounded set (whose existence is guaranteed by Theorem \ref{theorem pullback}).

Now for \eqref{eq:grownep}, there exists $\epsilon_0 >0$ such that
\[
\displaystyle \sup_{a_\epsilon \in \mathcal{A}_\epsilon(t)} \|S_\epsilon(t,\tau)a_\epsilon - S_0(t,\tau)a_\epsilon \| <  \frac{\delta}{2},
\]
for all $\epsilon < \epsilon_0$. Then,
\[
\begin{split}
{\rm dist}(\mathcal{A}_\epsilon(t),\mathcal{A}_0(t))  & \leqslant {\rm dist}(S_\epsilon(t,\tau) \mathcal{A}_\epsilon(\tau),S_0(t,\tau)\mathcal{A}_\epsilon(\tau)) +{\rm dist}(S_0(t,\tau) \mathcal{A}_\epsilon(\tau),S_0(t,\tau)\mathcal{A}_0(\tau)) \\
&  = \sup_{a_\epsilon \in \mathcal{A}_{\epsilon}(\tau)} {\rm dist}(S_\epsilon(t,\tau)a_\epsilon,  S_0(t,\tau)a_\epsilon) + {\rm dist}(S_0(t,\tau) \mathcal{A}_\epsilon(t),\mathcal{A}_0(t)) <  \frac{\delta}{2} +  \frac{\delta}{2},
  \end{split}
\]
which proves the upper-semicontinuity of the family of attractors.


\begin{thebibliography}{USA00}

\bibitem{Amann} H. Amann, {\it Linear and Quasilinear Parabolic Problems: Volume I: Abstract Linear Theory }, Monographs in Mathe\-ma\-tics, v.1, Birkhauser (1995).




\bibitem{TLR} T. Caraballo, G. Lukaszewicz, J. Real, Pullback attractors for asymptotically compact non-autonomous dynamical systems. \emph{Nonlinear Analysis} \textbf{64}, 484-498 (2006).

\bibitem{CCLR} T. Caraballo, A. N. Carvalho, J. A. Langa and F. Rivero, \emph{Existence of pullback attractors for pullback asymptotically compact process}, Nonlinear Analysis, 72, (3-4), 1967-1976, (2010).

\bibitem{CC1} A. N. Carvalho, J. W. Cholewa, \emph{Local well posedness for strongly damped wave equations with critical nonlinearities}, Bull. Aust. Math. Soc., 66 (3), 443-463 (2002).


\bibitem{CN} A. N. Carvalho, M. J. D. Nascimento, \emph{ Singularly non-autonomous
semilinear parabolic problems with critical exponents}, Discrete Contin. Dyn. Syst. Ser. S, Vol. 2, No. 3, 449-471 (2009).

\bibitem{Chen} S. Chen, D. L. Russell, {\it A mathematical model for linear elastic systems with structural damping}, Quarterly of Applied Mathematics, v. 39, \textbf{4}, 433-454 (1981).

\bibitem{Trig} S. Chen, R. Triggiani, {\it Proof of extensions of two conjectures on structural damping for elastic systems}, Pacific Journal of Mathematics, v. 136, \textbf{1}, 15-55 (1989).

\bibitem{ChV} V. V. Chepyzhov and M. I. Vishik \emph{ Attractors
for Equations of Mathematical Physics}, Providence, AMS Colloquium
Publications v. 49, A.M.S (2002).

\bibitem{DiBlasio} G. Di Blasio, K. Kunisch, E. Sinestrari, \emph{Mathematical models for the elastic beam with structural damping}, Appl. Anal. 48, 133-156 (1993).

\bibitem{EM} A. Eden, A. J. Milani, \emph{Exponential attractors for extensible beam equations}, Nonlinearity, v. 6, 457-479 (1993).



\bibitem{Haraux} A. Haraux, M. Ôtani, \emph{Analyticity and regularity for a class of second order evolution equations}, (preprint).

\bibitem{Haraux1} A. Haraux, \emph{Sharp estimates of bounded solutions to a second-order forced equation with structural damping}, Differential Equations \& Applications, 1, (3), 341-347 (2009).


\bibitem{Huang} F. Huang, \emph{On the mathematical model for linear elastic systems with analytic damping}, SIAM J. Control and Optimization, 126, (3) (1988).

\bibitem{Liu} K. Liu, \emph{Analyticity and Differentiability of Semigroups Associated with Elastic Systems with Damping and Gyroscopic Forces}, Journal of Diff. Equations, 141, 340-355 (1997).



\bibitem{Triebel} H. Triebel, \emph{Interpolation Theory, Function Spaces, Differential Operators}, North-Holland Pub. Co. (1978).


\bibitem{Xiao} T. Xiao, J. Liang, \emph{Semigroups Arising from Elastic Systems with Dissipation}, Computers Math. Applic., 33, (10), 1-9 (1997).


\bibitem{Zhong} C. Zhong, Q. Ma, C. Sun, \emph{Existence of strong solutions and global attractors for the suspension bridge equations}, Nonlinear Analysis, 67, 442-454 (2007).

\end{thebibliography}
\end{document}